\documentclass[reqno0,fleqn,12pt]{amsart}
\pdfoutput=1
\usepackage{dsfont}
\usepackage{bbm}
\usepackage[nodate]{datetime}
\usepackage[hmargin=30mm,top=30mm,bottom=40mm,a4paper]{geometry}
\usepackage{color}
\usepackage{subfig}
\usepackage{fancyhdr}
\usepackage{amsmath,amssymb,amsfonts,amsthm}
\usepackage{amstext}
\usepackage{amsmath}
\usepackage{amssymb}
\usepackage{enumerate}
\usepackage{amsbsy}
\usepackage{amsopn}
\usepackage{bbm,amsthm}
\usepackage{amscd}
\usepackage{amsxtra}
\usepackage{todonotes}
\usepackage{hyperref}

\newtheorem{theorem}{Theorem}[section]
\newtheorem*{theorem*}{Theorem}

\theoremstyle{remark}

\newcommand{\CC}{\mathds{C}}

\newcommand{\NN}{\mathds{N}}

\newcommand{\ind}{\mathds{1}}

\newcommand{\EE}{\mathbb{E}}

\begin{document}
\title{On multiplicative functions which are small on average and zero free regions for the Riemann zeta function}
\author{Marco Aymone}
\begin{abstract} In this short note we prove the following result: If a completely multiplicative function $f:\mathbb{N}\to[-1,1]$ is small on average in the sense that $\sum_{n\leq x}f(n)\ll x^{1-\delta}$, for some $\delta>0$, and if the Dirichlet series of $f$, say $F(s)$, is such that $F(1)=0$, then we obtain that for any $\epsilon>0$, $\sum_{p\leq x}(1+f(p))\log p\ll x^{1-\delta+\epsilon}$. Moreover, a necessary condition for the existence of such $f$ is that the Riemann zeta function $\zeta(s)$ has no zeros in the half plane $Re(s)>1-\delta$.

\end{abstract}
\maketitle

\section{Introduction}

We say that $f:\NN\to\CC$ is a multiplicative function if $f(nm)=f(n)f(m)$ whenever $\gcd(n,m)=1$, and we say that $f$ is completely multiplicative if this relation holds for all $n$ and $m$. 

A well known result in analytic number theory is that the Riemann zeta function $\zeta(s)$ is free of zeros in the half plane $Re(s)>1-\delta$, for some $0<\delta <1/2$, if and only if  $\sum_{n\leq x}\lambda(n)\ll x^{1-\delta+\epsilon}$, for all $\epsilon>0$, where $\lambda$ is the Liouville function (the completely multiplicative function that takes $-1$ at primes). Moreover, the Dirichlet series of $\lambda$ is $\zeta(2s)/\zeta(s)$ and hence has a zero at $s=1$. Therefore, under the Riemann hypothesis, $\lambda$ is an example of a completely multiplicative function such that the partial sums $\sum_{n\leq x}\lambda(n)\ll x^{1-\delta}$ for some $\delta>0$, and such that its Dirichlet series vanishes at $s=1$.

In this paper we are interested in the following question:

\noindent\textbf{Question:} \textit{Let $f:\NN\to[-1,1]$ be a completely multiplicative function and let $\mathcal{P}$ be the set of primes. If the partial sums $\sum_{n\leq x}f(n)\ll x^{1-\delta}$ for some $\delta>0$ and if the Dirichlet series $F(s)=\sum_{n=1}^\infty f(n)n^{-s}$ has a zero at $s=1$, then what is the mean behavior of the values $(f(p))_{p\in\mathcal{P}}$?}

In this short note we answer this question in the language of pretentious theory introduced by Granville and Soundararajan \cite{granvillepretentious}: Given two multiplicative functions $|f|,|g|\leq 1$, we say that $f$ is $g$-pretentious if
$$\sum_{p}\frac{1-Re(f(p)\bar{g}(p))}{p}<\infty.$$  
We show that any other example of a completely multiplicative function $f:\NN\to[-1,1]$ that satisfies the hypothesis of our question is $\lambda$-pretentious, and more than this: 

\begin{theorem}\label{teorema principal} Let $f:\NN\to[-1,1]$ be a completely multiplicative function such that $\sum_{n\leq x}f(n)\ll x^{1-\delta}$, for some $0<\delta<1/2$, and $F(s)=\sum_{n=1}^\infty\frac{f(n)}{n^s}$ satisfies $F(1)=0$. Then, $f$ is $\lambda$-pretentious and for all $\epsilon>0$
\begin{equation}\label{equacao teorema principal}
\sum_{p\leq x}(1+f(p))\log p\ll x^{1-\delta+\epsilon}.
\end{equation}
Moreover, if such function $f$ exists, then $\zeta(s)$ has no zeros in the half plane $Re(s)>1-\delta$.
\end{theorem}

In \cite{koukou} it has been proved (Theorem 1.6) that if $\sum_{n\leq x}f(n)\ll x^{1-\delta}$ for some $\delta>0$, and if the Dirichlet series $F(s)$ satisfies $F(1)=0$, then for some $0<\alpha\leq \frac{\delta}{61}$, the sum over primes $\sum_{p\leq x}(1+f(p))\log p\ll x^{1-\alpha}$. Thus the novelty here is the exponent $x^{1-\delta+\epsilon}$ and the necessary condition in which the half plane $Re(s)>1-\delta$ must be a zero free region for $\zeta(s)$.

The result above also improves the results in \cite{primeiropaper}, in which it has been proved a result of same quality of Theorem \ref{teorema principal} under randomness and bias assumptions. Indeed, if $(f(p))_p$ is a sequence of independent random variables with values in $\{-1,1\}$, and $f(n)$ is extended to all non-negative integers $n$ as a multiplicative function supported on the squarefree integers, then, under the bias assumptions $\EE f(p)<0$ and $\sum_{p}\frac{f(p)}{p}=-\infty$ almost surely, it has been proved (Theorems 1.2 and 1.4 of \cite{primeiropaper}) that the conclusions of Theorem \ref{teorema principal} holds almost surely. The novelty here is that we can obtain same conclusions under less restrictive conditions such as bias and randomness assumptions, and instead of almost all, we obtain conclusions for all multiplicative functions satisfying the hypothesis of Theorem \ref{teorema principal}.

\section{Proof of the main result}
 
 As always, $p$ denotes a generic prime number. We say that $f(x)\ll g(x)$ if there exists a constant $C>0$, such that $|f(x)|\leq C|g(x)|$ for all sufficiently large $x>0$.

\begin{proof}[Proof of Theorem \ref{teorema principal}] Let $h=1\ast f$. Then $h$ is multiplicative and for each prime $p$ and any power $m\in\NN$
\begin{equation}\label{equacao h nos primos}
h(p^m)=1+f(p)+f(p)^2+...+f(p)^m.
\end{equation}
If $f(p)\geq 0$, then by \eqref{equacao h nos primos}, $h(p^m)\geq 0$. If $-1\leq f(p)<0$, then by \eqref{equacao h nos primos}, $h(p^m)=\frac{1-f(p)^{m+1}}{1-f(p)}\geq 0$. Thus $h(p^m)\geq 0$ for all primes $p$ and all powers $m\in\NN$. Set $H(s):=\sum_{n=1}^\infty\frac{h(n)}{n^s}$. Now since $\sum_{n\leq x}f(n)\ll x^{1-\delta}$, we have that $F(s)=\sum_{n=1}^\infty \frac{f(n)}{n^s}$ converges and it is analytic in the half plane $Re(s)>1-\delta$. Moreover, as $F(1)=0$, we have that $H(s)=\zeta(s)F(s)$ is analytic in $Re(s)>1-\delta$, since the simple pole of $\zeta(s)$ at $s=1$ cancel with the (analytic) zero of $F(s)$ at $s=1$. Thus, $H(s)$ is a Dirichlet series of non-negative terms which is analytic in $Re(s)>1-\delta$. Thus, by the Landau's oscillation Theorem (see, for instance \cite{montgomerylivro}, pg. 16, Theorem 1.7), $H(s)$ converges in the half plane $Re(s)>1-\delta$.  Since the convergence is actually absolute, we have the convergence of the Euler product of $H(s)$ (see, for instance, \cite{tenenbaumlivro}, pg. 106, Remark of Theorem 2).  The convergence of this Euler product of $H(s)$ implies that for each $\sigma>1-\delta$ 
\begin{equation*}
\sum_{p}\sum_{m=1}^\infty \frac{h(p^m)}{p^{m\sigma}}<\infty.
\end{equation*}
In particular, $f$ is more than $\lambda$-pretentious:
\begin{equation*}
\sum_p \frac{h(p)}{p^\sigma}=\sum_p \frac{1+f(p)}{p^\sigma}<\infty.
\end{equation*}
Since the derivative of a convergent Dirichlet series is also convergent, 
\begin{equation*}
\sum_{p}\frac{(1+f(p))\log p}{p^\sigma}<\infty.
\end{equation*}
Set $\ind_{prime}(n)$ to be equal to $1$ if $n$ is prime and $0$ otherwise. Thus we have that the following series converges:
\begin{equation*}
\sum_{n=1}^\infty \frac{\ind_{prime}(n)(1+f(n))\log n}{n^\sigma}.
\end{equation*}
Now by Kroenecker's Lemma (see \cite{shiryaev}, pg. 390 Lemma 2), or by partial summation, we have that $\sum_{p\leq x}(1+f(p))\log p=\sum_{n\leq x}\ind_{prime}(n)(1+f(n))\log n=o( x^\sigma)$. This shows the first part of Theorem \ref{teorema principal}.

Now we are going to proof the second part of Theorem \ref{teorema principal}. We claim that the hypothesis $\sum_{n\leq x}f(n)\ll x^{1-\delta}$ implies that the series $\sum_{n=1}^\infty \frac{f(n)\mu^2(n)}{n^s}$ converges in the half plane $Re(s)>1-\delta$. Indeed, for $Re(s)>1$
\begin{align*}
F_{\mu^2}(s):=\sum_{n=1}^\infty \frac{f(n)\mu^2(n)}{n^s}&=\prod_p\left(1+\frac{f(p)}{p^s}\right)\\
&=\prod_p\left(1-\frac{f(p)^2}{p^{2s}}\right)\left(1-\frac{f(p)}{p^s}\right)^{-1}\\
&=F(s)\prod_p\left(1-\frac{f(p)^2}{p^{2s}}\right)\\
&:=F(s)U(s).
\end{align*}
Observe that $U(s)$ converges absolutely in $Re(s)>1/2$ and $F(s)$ converges in $Re(s)>1-\delta$. Thus $F_{\mu^2}(s)$ converges in $Re(s)>1-\delta$ (see for instance, \cite{tenenbaumlivro}, pg. 122, Notes 1.1). Now observe that,  $\frac{F_{\mu^2}}{F}(s)=U(s)$ is analytic and does not vanish in $Re(s)>1/2$. Hence, as $F(1)=0$, we obtain that $F_{\mu^2}(1)=0$. Set now $g=1\ast f\mu^2$. We have, for for any prime $p$ and any power $m\geq 1$:
\begin{equation*}
g(p^m)=1+f(p).
\end{equation*}
Thus, $g(p^m)\geq 0$ for all primes $p$ and all powers $m$, and since $G(s):=\sum_{n=1}^\infty\frac{g(n)}{n^s}=\zeta(s)F_{\mu^2}(s)$, with $F_{\mu^2}(s)$ being analytic in $Re(s)>1-\delta$ with $F_{\mu^2}(1)=0$, we obtain by the Landau's oscillation Theorem that $G(s)$ converges absolutely in $Re(s)>1-\delta$, and also the convergence of its Euler product. Now the Euler product of $G(s)$ is given by:
\begin{equation*}
G(s)=\prod_{p}\left(1+\sum_{m=1}^\infty \frac{1+f(p)}{p^{ms}} \right)=\prod_{p}\left(1+\frac{1+f(p)}{p^s-1} \right)=\prod_p\frac{p^s+f(p)}{p^s-1}.
\end{equation*}
Thus, each Euler factor above is $\neq 0$ in the half plane $Re(s)>1-\delta$, and hence $G(s)\neq 0$ in the half plane $Re(s)>1-\delta$. Thus, $\frac{1}{G(s)}$ is analytic in the half plane $Re(s)>1-\delta$, and since $\frac{1}{\zeta(s)}=\frac{F_{\mu^2}(s)}{G(s)}$, we obtain that $1/\zeta(s)$ is analytic in $Re(s)>1-\delta$. \end{proof}

%\bibliography{ay}{}
%\bibliographystyle{siam}

{\small{\sc \noindent Marco Aymone \\
Departamento de Matem\'atica, Universidade Federal de Minas Gerais, Av. Ant\^onio Carlos, 6627, CEP 31270-901, Belo Horizonte, MG, Brazil.} \\
\textit{Email address:} aymone.marco@gmail.com}
\vspace{0.5cm}

\end{document}